\nonstopmode
\documentclass[12pt]{amsart}
\usepackage{amscd}
\usepackage{amssymb}

\input xy
\xyoption{all}

\numberwithin{equation}{section}

\swapnumbers

\theoremstyle{plain} 
\newtheorem{thm}[equation]{Theorem}

\newtheorem{lem}[equation]{Lemma}
\newtheorem{prop}[equation]{Proposition}

\theoremstyle{definition}

\theoremstyle{remark}
\newtheorem{rem}[equation]{Remark}

\newtheorem*{VariableNoNum}{{\VariableText}}
\newtheorem{Variable}[equation]{{\VariableText}}

\theoremstyle{definition}
\newtheorem*{VariableNoNumBold}{{\VariableText}}
\newtheorem{VariableBold}[equation]{{\VariableText}}

\newenvironment{titled}[1]
     {\def\VariableText{{#1}}\begin{VariableNoNum}}
     {\end{VariableNoNum}}

\newenvironment{numbered}[1]
     {\def\VariableText{{#1}}\begin{Variable}}
     {\end{Variable}}

\newlength{\asidelength}

\def\Changed/{\ifvmode\else\vadjust{%
\vbox to 0pt{\vskip -\baselineskip%
\hbox to 0pt{\hss\vrule height 0pt depth 1.2\baselineskip\hskip 1em}\vss}}\fi}
\def\CHanged{\ifvmode\else\vadjust{%
\vbox to 0pt{\vskip -\baselineskip%
\hbox to 0pt{\hss\vrule height 0pt depth 1.2\baselineskip\hskip 1em}\vss}}\fi}

\def\Math#1{\def\MathString{#1}\futurelet\MathDelim\MathChoose}
\def\MathChoose{\ifmmode\let\MathDo\MathString%
              \else\let\MathDo\MathSkip\fi%
              \MathDo}
\def\MathSkip{\ifx\MathDelim/\def\MathDo{$\MathString$\EatOne}%
              \else\def\MathDo{$\MathString$}\fi%
              \MathDo}

\def\Text#1{\def\TextString{#1}\futurelet\TextDelim\TextSkip}
\def\TextSkip{\ifx\TextDelim/\def\TextDo{\TextString\EatOne}%
              \else\let\TextDo\TextString\fi%
              \TextDo}
\def\EatOne#1{}

\def\SkipToEndScan#1\EndScan{}
\def\Scan#1#2#3{\ifx#1#2#3\expandafter\SkipToEndScan\fi\Scan#1}
\def\Upper#1{%
\Scan#1aAbBcCdDeEfFgGhHiIjJkKlLmMnNoOpPqQrRsStTuUvVwWxXyYzZ#1#1\EndScan}
\def\Phrase#1 #2/#3/#4=#5 #6/#7/#8.{%
\expandafter\edef\csname#2#3\endcsname{\noexpand\Text{#6#7}}
\expandafter\edef\csname\Upper#2#3\endcsname{\noexpand\Text{\Upper#6#7}}
\expandafter\edef\csname#1#2#3\endcsname{\noexpand\Text{#5 #6#7}}
\expandafter\edef\csname\Upper#1#2#3\endcsname{\noexpand\Text{\Upper#5 #6#7}}
\expandafter\edef\csname#2#4\endcsname{\noexpand\Text{#6#8}}
\expandafter\edef\csname\Upper#2#4\endcsname{\noexpand\Text{\Upper#6#8}}
}

\newcommand{\whatever}{\text{--}}

\newcommand{\op}{^{\text{op}}}

\newcommand{\Z}{\Math{\mathbb{Z}}}

\newcommand{\Fp}{\Math{\mathbb{F}_p}}
\newcommand{\Zp}{\Math{{\Z}\hat{_p}}}
\newcommand{\Zmodp}{\Math{\Z/p}}

\newcommand{\Tensor}{\otimes}

\newcommand{\RightArrow}[1]{\xrightarrow{#1}} 

\newcommand{\Hom}{\operatorname{Hom}}

\newcommand{\hhh}{\operatorname{h}\!}
\newcommand{\thhh}{\tilde{\operatorname{h}}}
\def\HomotopyOrbit#1on#2/{\ensuremath{#2_{\hhh#1}}}
\def\RedHomotopyOrbit#1on#2/{\ensuremath{#2_{\thhh#1}}}

\Phrase a dGA//s=a dGA//s.
\Phrase an equivalence//s=an equivalence//s.
\newcommand{\ring}{\Math{R}}
\newcommand{\dga}{\Math{A}}
\newcommand{\dgab}{\Math{B}}
\newcommand{\dgm}{\Math{X}}
\newcommand{\dgmb}{\Math{Y}}
\newcommand{\dgmc}{\Math{Z}}
\newcommand{\Br}[1]{\Math{\beta_{{#1}}}}
\newcommand{\BrOf}[2]{\Math{\beta_{{#1}}(#2)}}
\newcommand{\PrOf}[2]{\Math{P_{{#1}}(#2)}}
\newcommand{\Bp}{\Br{\Fp}}
\newcommand{\dvr}{\Math{\mathcal O}}
\newcommand{\ideal}{\Math{\mathfrak m}}
\newcommand{\pwr}[1]{{[}{[}#1{]}{]}}
\newcommand{\poly}[1]{{[}#1{]}}
\newcommand{\unif}{\Math{\pi}}

\newcommand{\pzero}{\Math{(\Z/p,0)}}

\newcommand{\EE}{\Math{\mathcal{E}}}
\newcommand{\End}{\operatorname{End}}
\newcommand{\Ext}{\operatorname{Ext}}
\newcommand{\LL}{\Math{\Lambda}}

\newcommand{\Gr}{\operatorname{Gr}}
\newcommand{\ang}[1]{\langle#1\rangle}

\newcommand{\mat}[4]{\left( 
              \begin{array}{cc}
               #1&#2\\
               #3&#4
              \end{array}
                    \right)}

\newcommand{\FF}{\Math{\mathcal F}}
\newcommand{\NN}{\Math{N}}
\newcommand{\DD}{\Math{D}}
\Phrase an admissible//s=a distinguished//s.
\Phrase a distinguished//s=a distinguished//s.
\newcommand{\CC}{J}
\newcommand{\hocolim}{\operatorname{hocolim}}
\renewcommand{\op}{^{\operatorname{op}}}
\begin{document}
\title[Exterior homology]{DG algebras with exterior homology}
\author {W. G. Dwyer, J. P. C. Greenlees and S. B. Iyengar}

\address{Department of Mathematics, University of Notre Dame, Notre
  Dame, Indiana 46556, USA} 
\email{dwyer.1@nd.edu}
\address{Department of Pure Mathematics, Hicks Building, Sheffield S3
  7RH, UK}
\email{j.greenlees@sheffield.ac.uk}
\address{Department of Mathematics, University of Nebraska,
  Lincoln, NE 68588, USA}
\email{iyengar@math.unl.edu}
\date{\today}
\thanks{ WGD was partially supported by NSF grant DMS~0967061, JPCG by
  EPSRC grant EP/HP40692/1, and SBI 
  by NSF grants DMS~0903493 and DMS~1201889.}
\begin{abstract}
We study  differential graded
algebras (\DGAs) whose homology is an  exterior algebra over a
commutative ring $R$ on a generator of degree $n$, and also certain
types of differential modules over these \DGAs. We obtain a complete
classification with $R=\Z$ or $R=\Fp$ and $n\ge-1$. The
examples are unexpectedly interesting.
\end{abstract}

\maketitle

\section{Introduction}\label{CIntro}
We mainly study \DGAs/ whose homology is an exterior algebra
over either \Fp/ or \Z/ on a class of degree~$-1$, as well as (left) modules
over these \DGAs/ whose homology
is \Zmodp/ in degree~0. In the case of an exterior algebra over \Fp, we find
many \DGAs/, each having one such module; in the case of an exterior
algebra over \Z, one \DGA/
having many such modules. In both cases, the enumeration of
possibilities involves complete discrete valuation rings with
residue field~\Fp. There are remarks about other types of \DGAs/ in
\S\ref{COther}. 

In more detail, a \emph{differential graded algebra} (\DGA) is a chain complex $\dga$ of
abelian groups together with a
multiplication map $\dga\Tensor \dga\to \dga$ which is both unital and
associative.   A morphism $f\colon \dga\to \dgab$ of \DGAs/ is a map
of chain complexes which respects  multiplication and  unit; $f$~is said to be \emph{\anequivalence/} if it induces isomorphisms $H_i\dga\cong H_i\dgab$,
$i\in\Z$. The homology $H_*\dga$ is a graded ring, and
$\dga$ is said to be \emph{of type \Br S/}, for $S$ a
commutative ring,  if $H_*\dga$ is an
exterior algebra over $S$ on a class of degree~$-1$.  
The notation  \Br S/ is meant to suggest that \dga/
captures a type of Bockstein operation. In topology, the ``Bockstein'' is
a operation $\beta$ on  mod~$p$ homology, of degree~$-1$ and square~0, which
arises from a generator of $\Ext^1_{\Z}(\Z/p,\Z/p)$.

A \emph{module}
over $\dga$ is a chain complex $\dgm$ together with an action map $\dga\Tensor
\dgm\to \dgm$ with the usual unital and associativity
properties. Morphisms and equivalences between modules are defined
in the evident way.
The module $\dgm$ is of \emph{type $(M,0)$} for an abelian group
$M$ if there are isomorphisms of abelian groups
\[
    H_i\dgm\cong\begin{cases} M & i=0\\
                           0&\text{otherwise}
             \end{cases}
\]

We are particularly interested in \DGAs/ of type $\Br S$ for $S=\Fp$
or $S=\Z$, and modules over these \DGAs/ of type $\pzero$.

\begin{thm}\label{Modpcase}
 There is a natural bijection
 between 
 \begin{itemize}
  \item equivalence classes of \DGAs/ of type $\Bp$, and 
  \item isomorphism classes of complete
 discrete valuation rings with residue  field~$\Fp$.
 \end{itemize}  Up to
 equivalence, each such \DGA/ has a unique module of type $\pzero$.
\end{thm}

\begin{rem}
  A \emph{complete discrete valuation ring} \dvr/ is a principal ideal domain
  with a unique nonzero prime ideal \ideal/, such that $\dvr$ is
  complete with respect to the topology determined by powers
  of~\ideal. If $\dvr/\ideal\cong\Fp$, then $\dvr$ is isomorphic
  either to $\Fp\pwr t$  or to a totally ramified
extension of finite degree of the ring \Zp/ of $p$-adic integers; see \cite[II.4-5]{Serre}.
The surprise in \ref{Modpcase} is the sheer profusion of
\DGAs. In \ref{ModelForDGA} below there is an explicit description of
how to pass from a ring \dvr/ to a \DGA.
\end{rem}

The tables are turned when it comes to \DGAs/ of type $\Br \Z$.

\begin{thm}\label{ModZcase}
  Up to equivalence there is only a single
  \DGA/ \dga/ of type $\Br \Z$. There is a natural
  bijection between
  \begin{itemize}
  \item equivalence classes of modules over $\dga$ of type $\pzero$, and
  \item isomorphism classes of pairs $(\dvr,\unif)$, where $\dvr$ is a
    complete discrete valuation ring with residue field $\Fp$,
    and $\unif$ is a uniformizer for $\dvr$.
  \end{itemize}
\end{thm}

\begin{rem}
A \emph{uniformizer} for \dvr/  is a generator of the maximal
  ideal~\ideal. 
  The surprise in \ref{ModZcase} is  the profusion of
  modules, since there are many pairs $(\dvr,\unif)$ as above. For
  instance, if $\pi\in\Zp$ is divisible by $p$ but not by $p^2$, then
  $(\Zp,\pi)$ is such a pair. Since \Zp/ has no nontrivial
  automorphisms, these
  pairs are distinct for different choices of $\pi$. 

  The object \dga/ of \ref{ModZcase}  can be taken to be the \DGA/ \FF/
which contains copies of $\Z$ in degrees~0 and $-1$, is trivial
elsewhere, and has zero differential.  In \ref{ChainComplexFromModule} below there is
  an explicit description of how to pass from a pair $(\dvr,\unif)$ to
  a module over \FF. 
\end{rem}

\begin{rem}\label{SometimesUnique}
  If \dga/ is a \emph{connective} \DGA, i.e., $H_i\dga=0$ for $i<0$, and
  $\dgm$ is a module over \dga/ such that $H_i\dgm=0$ for $i\ne 0$,
  then $\dgm$ is determined up to equivalence by the isomorphism class
  of $H_0\dgm$ as an (ordinary) module over $H_0\dga$. See for
  instance \cite[3.9]{Duality}. This is in strong contrast to what
  happens in the non-connective setting of~\ref{ModZcase}.
\end{rem}

\begin{titled}{Generalizations} We have some remarks in \S\ref{COther}
  about \DGAs/ with other types of exterior algebra homology.

  The arguments below can be interpreted in the setting of stable
  homotopy theory, and they lead to a classification of associative
  ring spectra of type \Br \Z/ or \Bp/ (appropriately interpreted) and of
  module spectra over these ring spectra of type \pzero. No new
  examples come up; all of these ring spectra and module spectra are
  obtained in a standard way \cite{ShipleyHZ}  from \DGAs. 
\end{titled}

\begin{numbered}{DG-objects, equivalences, and formality}\label{Formality}
As suggested above, a map between differential graded (DG) objects of
any kind is said to be an \emph{equivalence} if it induces an
isomorphism on homology groups. Two objects are \emph{equivalent} if
they are related by a zig--zag
$\leftarrow\rightarrow\leftarrow\cdots\rightarrow$ of equivalences. A
DG-object $X$ is \emph{formal} if it is equivalent to a DG-object $Y$ of
the same kind which has zero differentials. Of course, in this case
the graded constituent $Y_i$ of $Y$ must be isomorphic to $H_iX$. For
instance (cf. \ref{SometimesUnique}), if \dga/ is a connective \DGA/
and $X$ is a module over~\dga/ such that $H_iX$ vanishes except for a
single value of~$i$, then $X$ is formal as an \dga-module. 

If \dga/ is a \DGA/ such that $H_i\dga=0$ for $i\ne 0$,
then \dga/ is formal as a \DGA/. To see this, let $\dga'\subset\dga$ be the
subcomplex given by
\[
\dga'_i=\begin{cases} \dga_i & i>0\\
                       \ker(\partial\colon\dga_0\to\dga_{-1}) & i=0\\
                       0 &i<0
        \end{cases}
\]
Then $\dga'$ is a \DGA/, and there is a zig-zag of \DGA-equivalences
\[
   \xymatrix@1{\dga &\dga'\ar[r]^\sim\ar[l]_\sim & H_0\dga}
\]
where the object on the right is a ring, treated as a \DGA/
concentrated in degree~0.
 
 If  \dga/ is a \DGA/ such that $H_*\dga$ is a polynomial
  algebra $\Z[x]$ on a class $x$ of (arbitrary) degree~$n$, then \dga/
  is also formal: choose a cycle
  $\chi\in\dga_n$ representing~$x$ and construct an equivalence  $(\Z[x],0)\to(\dga,\partial)$  by $x\mapsto \chi$. On the other hand, if $S$ is a commutative
  ring other than \Z/ and $H_*\dga\cong S[x]$, it is not necessarily
  the case that \dga/ is formal, even for $S=\Fp$ (see \S\ref{COther}
  for examples); the above argument
  applies only if \dga/ itself is an algebra over~$S$, or at
  least equivalent to an algebra over~$S$. 
\end{numbered}

\begin{numbered}{Relationship to Moore-Koszul duality}\label{MooreKoszul}
  Suppose that $k$ is some chosen field. Say that a \DGA/
  \dga/ over $k$ is \emph{admissible} if  $H_i\dga$ is finite-dimensional over~$k$
  for all $i$ and $H_0\dga\cong k$. Applying duality over~$k$ to the
  bar constructions in \cite{rMoore} produces bijections between
  equivalence classes of the  admissible \DGAs/ indicated below.
\[
\xymatrix@1@C=4pc{\txt<13pc>{%
        \{$\dga\vert H_i\dga=0$, $i>0$ \& $i=-1$\}}\quad \ar@<0.5ex>[r]^-{\dga\mapsto\End_{\dga}(k)} &\txt<10pc>{\{%
      $\dgab\vert H_i\dgab\cong 0$, $i<0$\}}
      \quad\ar@<1.5ex>[l]^-{\dgab\mapsto\End_{\dgab}(k)}}
\]
In topology, for instance, this gives the relationship between the
cochain algebra \dga/ of a simply-connected space $X$ and the chain algebra \dgab/ of
the loop space $\Omega X$.

Our technique is to push the boundaries of the above  Moore-Koszul duality 
construction. In proving \ref{Modpcase}
(\S\ref{CSemiCanon}-\ref{CExtOverFp}),  we start with a \DGA/ \dga/ of type \Bp, more or
less construct by hand an action of \dga/ on (something equivalent to)
$\Z/p$, and show that the \DGA/ $\dgab=\End_{\dga}(\Z/p)$ has its
homology concentrated in degree~0, and so is essentially an ordinary
ring \dvr/ (\ref{Formality}). It turns out that  \dvr/ is a complete discrete valuation
ring, and that it determines \dga/
up to equivalence via the formula
\[
        \xymatrix@1{\dga\ar[r]^-\sim&\End_{\dvr}(\Z/p)}\,.
\]
        Note that it is \emph{not} necessarily the case
that \dga/ is a \DGA/ over~\Fp. In \S\ref{CUniqueOverZ}, the same technique
succeeds in classifying \DGAs/ of type \Br\Z/. The proof in
\S\ref{CManyModules} of the second part   of \ref{ModZcase} relies on
ideas from~\cite{CompleteTorsion} which fit modules
into a type of Moore-Koszul duality setting.
\end{numbered}

\begin{numbered}{Notation and terminology}\label{Notation}
We work in the context of \cite{CompleteTorsion} and
\cite{Duality}. Rings are tacitly identified with \DGAs/ concentrated
in degree~0 (\ref{Formality}). An ordinary module $M$ over a ring $R$ is similarly
tacitly identified with the chain complex $X$ over $R$ with $X_0=M$
and $X_i=0$ for $i\ne0$. $\Hom$ is the derived homomorphism
complex and $\Tensor$ is the derived tensor product. If \dgm/ is a
module over the \DGA/ (or ring) \dga, then  $\End_{\dga}(\dgm)$
denotes the  \DGA/ obtained by taking a projective model for \dgm/
  and forming the usual \DGA/ of endomorphisms of this model.  See
  \cite[2.7.4]{rWeibelBook}, but re-index to conform to our convention that
  differentials always reduce degree by one.  (Up to
  equivalence, the \DGA/ $\End_{\dga}(X)$ depends only on the
  equivalence types of \dga/ and of $X$; this can be proved for
  instance with the bimodule argument of \cite[3.7.6]{rSchwedeShipleyModules}.)
If \ring/ is a ring and $M$ is an ordinary $R$-module, there are
isomorphisms
\[
       H_i\End_R(M)\cong \Ext^{-i}_R(M,M)\,.
\]
We write $\Ext_0^R(M,M)$ or $H_0\End_R(M)$ for the ordinary endomorphism
ring of $M$ over $R$.

If $X$ is a chain complex or graded abelian group, we write
$\Sigma^iX$ for its $i$-fold shift: $(\Sigma^i X)_j = X_{j-i}$,
$\partial(\Sigma^ix)=(-1)^i\Sigma^i(\partial x)$. For the sake of clarity, we attempt as much as reasonably possible to
make a notational distinction between the field $\Fp$ and the abelian
group $\Z/p=\Z/p\Z$.
\end{numbered}
\section{A module of type $(\mathbb Z/p,0)$ }\label{CSemiCanon}

In this section \dga/ is a \DGA/ of type $\Bp$. We construct a module
\dgm/ over \dga/ of type $\pzero$ and show that for any module
$\dgmb$ over \dga/ of type $\pzero$ there is \anequivalence/
$e_{\dgmb}:\dgm\to\dgmb$. However, the \equivalence/ $e_{\dgmb}$ is not
unique in any sense, even up to homotopy. See \cite[3.3, 3.9]{Duality}
for more general constructions of this type.

The construction is inductive. Suppose that \dgmc/ is a module over
\dga/ with 
\begin{equation}\label{SpecialModule}
  H_i\dgmc \cong\begin{cases} \Z/p & i=0, -1\\
                                0 & \text{otherwise}
                 \end{cases}
\end{equation}
Choose a map $\kappa:\Sigma^{-1}\dga\to\dgmc$ which induces an isomorphism
\[
   \Z/p\cong H_0\dga \cong H_{-1}\Sigma^{-1}\dga\RightArrow{H_{-1}\kappa} H_{-1}\dgmc
\]
and let $\CC\dgmc$ be the mapping cone of $\kappa$. There is a natural
map $\dgmc\to \CC\dgmc$ which is an isomorphism on $H_0$ and trivial on
the other homology groups (in particular $H_{-1}\dgmc\to
H_{-1}\CC\dgmc$ is zero), and so the homology exact sequence of the
triangle
\[
      \Sigma^{-1}A\to\dgmc\to \CC\dgmc
\]
shows that $H_*\CC\dgmc$ again  vanishes except for copies of $\Zmodp$ in
degrees $0$ and $-1$.  Starting with $\dgmc=\dga$, iterate the process
to obtain a sequence
\[
   \dga\to \CC\dga\to \CC^2\dga\to\cdots
\]
of chain complexes and chain maps, and let
$\dgm=\hocolim_k\CC^k\dga$. (In this case the homotopy colimit can be
taken to be a colimit, i.e., the ascending union.) It is immediate that $X$ has type
$\pzero$.

Suppose that $\dgmb$ is an arbitrary module of type \pzero. There is
certainly a map $\dga\to\dgmb$ which induces an isomorphism on $H_0$,
so to construct \anequivalence/ $\dgm\to\dgmb$ it is enough to show
that if \dgmc/ satisfies \ref{SpecialModule} and $f\colon\dgmc\to \dgmb$
induces an isomorphism on $H_0$, then $f$ extends to $f'\colon
\CC\dgmc\to\dgmb$.  (By induction, this will guarantee that the map $\dga\to\dgmb$
extends to a map $\dgm=\hocolim_k\CC^k\dga\to\dgmb$.) The map $f$ extends to $f'$ if and only if the composite
\[
    \xymatrix@1{ {\hbox{$\Sigma^{-1}\dga$}} \ar[r]^-{\kappa}& \dgmc\ar[r]^{f} &\dgmb}
\]
is null homotopic. But the group of homotopy classes of \dga-module
maps $\Sigma^{-1}\dga\to\dgmb$ vanishes, since it is isomorphic to
$H_{-1}\dgmb$.

\section{Exterior algebras over $\Fp$}\label{CExtOverFp}

In this section we prove \ref{Modpcase}. The proof depends on two lemmas.

\begin{titled}{Obtaining a ring from a \DGA/} 
Suppose that \dga/ is a \DGA/
of type $\Bp$. According to \S\ref{CSemiCanon}, up to equivalence there is a unique
module \dgm/ over \dga/ of type \pzero. Let $\EE=\End_{\dga}(\dgm)$ be
the derived endomorphism algebra of $\dgm$.

\begin{lem}\label{DGAtoRing} Let \dga, \dgm, and \EE/ be as above.
  \begin{enumerate}
  \item The homology group $H_i\EE$ vanishes for $i\ne0$, and the ring
  $H_0\EE$ is a complete discrete valuation ring with residue
  field $\Fp$.
  \item The natural  map $\dga\to\End_{\EE}(X)$ is \anequivalence.
 \end{enumerate}
\end{lem}
\end{titled}

\begin{titled}{Obtaining a \DGA/ from a ring}
Suppose that \dvr/ is a complete discrete valuation ring with residue class
field \Fp. Let $X$ denote  $\Zmodp$ with the unique possible
\dvr-module structure, and let
$\dga=\End_{\dvr}(X)$ be the derived endomorphism algebra of $X$. 

\begin{lem}\label{RingtoDGA}
  Let \dvr, \dgm, and \dga/ be as above.
  \begin{enumerate}
  \item \dga/ is a \DGA/ of type \Bp.
  \item The natural map $\dvr\to\End_{\dga}(X)$ is an equivalence.
  \end{enumerate}
\end{lem}
\end{titled}

\begin{titled}{Proof of \ref{Modpcase}}
  The existence and uniqueness of the module of type \pzero/ is from
  \S\ref{CSemiCanon}. For the rest, \ref{Formality}, \ref{DGAtoRing} and \ref{RingtoDGA}
  provide inverse constructions matching up appropriate \DGAs/ with
  appropriate rings. \qed
\end{titled}

For minor efficiency reasons, we first prove \ref{RingtoDGA} and then \ref{DGAtoRing}.

\begin{numbered}{Proof of \ref{RingtoDGA}}\label{DGAisBockstein}
For part (1), observe that as usual there are isomorphisms
  \[
      H_i\End_{\dvr}(X)\cong
                         \Ext^{-i}_{\dvr}(\Zmodp,\Zmodp)\,.
\] 
But there is a short projective resolution of $\Zmodp$ over \dvr/ 
\begin{equation}\label{ShortResolution}
   \xymatrix@1{0\ar[r] &\dvr\ar[r]^{\unif}&\dvr\ar[r]&\Zmodp\ar[r]&0}\,.
\end{equation}
By inspection, then,  $\Ext^i_{\dvr}(\Zmodp,\Zmodp)$ vanishes unless $i=0$ or
$i=1$, and in these two exceptional cases the group is isomorphic to
$\Zmodp$. 

Since $\dvr/\ideal\cong\Fp$ is a field and hence a regular ring,
\ref{RingtoDGA}(2) 
is \cite[4.20]{Duality}. \qed
\end{numbered}

\begin{rem}\label{ModelForDGA}
  If \dvr/ is as in
  \ref{RingtoDGA}, an explicit  model for 
  $\End_{\dvr}(\Zmodp)$ can be derived from \ref{ShortResolution} as follows.
\[
\xymatrix@C=6pt{1 & \ang{L}\ar[d]^{\partial L=\unif D_1+\unif D_2}\\
           0 & \ang{D_1,D_2}\ar[d]^{\partial D_i=(-1)^i\unif U}\\
           -1 & \ang{U}}
\quad\qquad\raisebox{-40pt}{\hbox{$\displaystyle \begin{aligned}
   D_1&=\mat{1}{0}{0}{0}\quad\quad 
  \phantom{\hbox{${}_2$}}U=\mat{0}{1}{0}{0} \\
  L&=\mat{0}{0}{1}{0}\quad\quad
  D_2=\mat{0}{0}{0}{1}.
    \end{aligned}$}}
\]
This is a \DGA/ which is nonzero only in degrees $1,\,0,\,-1$; the
notation $\ang{\whatever}$ denotes the free \dvr-module on the
enclosed generators. From a multiplicative point of view the \DGA/
is as indicated a graded form of the ring of $2\times 2$ matrices over~\dvr.
\end{rem}

\begin{titled}{Proof of \ref{DGAtoRing}(1)}
 Let $\LL$ denote the graded algebra $H_*\dga$ and $x\in\LL_{-1}$ an
  additive  generator. Write $\Zmodp$ for the (ordinary) graded \LL-module
  $\LL/\langle x\rangle\cong H_*\dgm$. 
  For general reasons (see \ref{ExplainEM} below)  there is a left half plane Eilenberg-Moore spectral
  sequence
  \[
    E^2_{-i,j}=  \Ext^i_{\LL}(\Sigma^j\Zmodp,\Zmodp)\Rightarrow H_{j-i}\EE\,.
  \]
  (Note that $\Ext$ here is computed in the category of graded modules
  over a graded ring.)
  This is a conditionally convergent spectral sequence of bigraded algebras with differentials
  \[d_r:E^r_{i,j}\to E^r_{i-r,j+r-1}\,,\] and it  abuts to the
  graded algebra $H_*\EE$.
  There is a free resolution of
  \Zmodp/ over \LL/
  \[
      \xymatrix@1{\cdots\ar[r]^-x&{\hbox{$\Sigma^{-2}\LL$}}\ar[r]^-x&\Sigma^{-1}\LL\ar[r]^-x&\LL\ar[r]&\Zmodp\ar[r]&0}
  \]
  which leads by calculation to the conclusion that the bigraded algebra
  $\Ext^i_{\LL}(\Sigma^j\Zmodp,\Zmodp)$ is a polynomial algebra 
  on the extension class
  \[
      \xymatrix@1{0\ar[r] &
         \Zmodp\ar[r]&\Sigma\LL\ar[r]&\Sigma\Zmodp\ar[r]&0}
  \]
  in $     \Ext^1_{\LL}(\Sigma\Zmodp,\Zmodp)$.
  It follows that 
$E_{i,j}^2$ vanishes in the above spectral sequence for
  $i\ne-j$, the spectral sequence collapses, and  
  $E^\infty=E^2$ is concentrated in total degree~0. Hence
  $H_i\EE$ vanishes for $i\ne0$, and $H_0\EE$ is a ring $\dvr$ with
  a decreasing sequence of ideals
  \[
    \cdots\subset
    \ideal_k\subset\cdots\subset\ideal_2\subset\ideal_1\subset\dvr 
  \]
  such that $\ideal_k\ideal_\ell\subset\ideal_{k+\ell}$,
  $\Gr(\dvr)\cong\Fp[t]$, and $\dvr\cong\lim_k\dvr/\ideal_k$.  

  Let $\unif\in\ideal_1\subset\dvr$ be an element which projects to a
  generator of $\ideal_1/\ideal_2\cong\Zmodp$, and map $\Z\poly s$ to
  \dvr/ by sending $s$ to $\unif$. It is easy to argue by induction on $k$
  that the composite $\Z\poly s\to \dvr\to\dvr/\ideal_k$ is surjective,
  so that $\dvr/\ideal_k$ is commutative and
  $\dvr\cong\lim_k\dvr/\ideal_k$ is commutative as well. The ring
  \dvr/ is a noetherian domain because it has a complete filtration
  $\{\ideal_k\}$  such that
  $\Gr(\dvr)$ is a noetherian domain \cite[Chap.~III, \S2, Corr. 2]{Bourbaki}. The ideal
  $\ideal_1/\ideal_k$ is nilpotent in $\dvr/\ideal_k$, and 
  hence  an element $x\in\dvr/\ideal_k$ is a unit if and only the image
  of $x$ in $\dvr/\ideal_1\cong\Fp$ is nonzero. It follows directly
  that $x\in\dvr$ is a unit if and only if the image of $x$ in
  $\dvr/\ideal_1$ is nonzero, and so $\dvr$ is a local ring with
  maximal ideal $\ideal=\ideal_1$. Finally, by induction on $k$ there
  are exact sequences
  \[
    \xymatrix@1{\dvr/\ideal_k \ar[r]^{\unif}&\dvr/\ideal_k\ar[r]&\dvr/\ideal\ar[r]&0}
  \]
  The groups involved are finite, so passing to the limit in $k$ gives
  an exact sequence
  \[
  \xymatrix@1{\dvr\ar[r]^{\unif} &\dvr \ar[r]&\dvr/\ideal\ar[r]&0}
  \]
  expressing the fact that $\unif$ generates the ideal $\ideal$. The
  element $\unif$ is not nilpotent (because its image in
  $\Gr(\dvr)\cong\Fp\poly t$ is $t$) and so by Serre
  \cite[I.\S2]{Serre} \dvr/ is a discrete valuation ring.
  \qed
\end{titled}

\begin{titled}{Proof of \ref{DGAtoRing}(2)} By \ref{DGAisBockstein}
  and \ref{DGAtoRing}(1), $\End_{\EE}(X)$ is a \DGA/ of type \Bp.
  It is thus enough to show that the map $H_{-1}\dga\to
  H_{-1}\End_{\EE}(X)$ is an isomorphism, or even that this map is
  nonzero.  We will use the notation of \S\ref{CSemiCanon}. Note that
  the mapping cone of the \dga-module map $\epsilon:\dga\to \CC\dga$ is again $\dga$,
  and that the mapping cone of any nontrivial map $\dga\to\dgm$ is
  again equivalent to \dgm, this last by a
  homology calculation and the uniqueness result of
  \S\ref{CSemiCanon}.
  Consider the following diagram, in which the rows are exact
  triangles and the middle vertical map is provided by \S\ref{CSemiCanon}.
  \begin{equation}\label{Cofibrations}
    \xymatrix{
        \dga \ar[r]\ar[d]^= &\CC\dga\ar[r]\ar[d]& \dga\ar[d]^{\text{``$\epsilon$''}}
        \ar@{-->}[r]^-a &\Sigma\dga\\
         \dga\ar[r]^\epsilon &\dgm\ar[r]^{\unif}&\dgm}
  \end{equation}
  Here $a$ denotes (right) multiplication by a generator of
  $H_{-1}\dga$. (In the language of \cite{Duality}, the lower row here
  shows that $X$ is proxy-small over \dga, and \ref{DGAtoRing}(2)
  now follows fairly directly from \cite[pf.~of~4.10]{Duality}. For convenience we
  continue with a direct argument.)
  After a suitable identification of the mapping cone of
  $\epsilon$ with $X$, a homology calculation gives that the  right
  vertical map is homotopic to~$\epsilon$.
  The
  right lower map is entitled to be labeled $\unif$, as in
  \ref{DGAisBockstein}, because the
  homology of its mapping cone (namely $H_*\Sigma\dga=\Sigma H_*A$) evidently represents a
  nonzero element of 
  \[
     \Ext^1_{H_*\dga}(\Sigma\Zmodp,\Zmodp)\,.
  \]
  Note that the element \unif/ of \dvr/ is determined only up to
  multiplication by a unit, and this is reflected above in the fact
  that there are  various
  ways to identify the  mapping cone of~$\epsilon$ with~\dgm. Applying
  $\Hom_{\dga}(\whatever,X)$ to \ref{Cofibrations} gives
  \begin{equation}\label{MappedCofibrations}
    \xymatrix{
      X & \ar[l] \Hom_{\dga}(\CC\dga,\dgm) & \ar[l] X &\ar@{-->}[l]_{a^*} \Sigma^{-1}X\\
      X\ar[u]^=  & \ar[l]_{\epsilon^*}\ar[u] \EE & \ar[u]_{\epsilon^*}\ar[l]_{\unif^*} \EE}
  \end{equation}
  The map $\epsilon^*$ is surjective on homology, since
  by the argument of \S\ref{CSemiCanon} any \dga-map $\dga\to\dgm$
  extends over~$\epsilon$ to a map $\dgm\to\dgm$. Let $M$ be the
  \dvr-module $H_0\Hom_{\dga}(\CC\dga,\dgm)$. Applying $H_0$ to the
  solid arrows in  diagram
  \ref{MappedCofibrations} gives a diagram of exact sequences
  \[
  \xymatrix{0& \ar[l]\Zmodp &\ar[l] M{\vphantom{\hbox{$\Zmodp$}}} & \ar[l] \Zmodp&\ar[l] 0 \\
            0 & \ar[l] \Zmodp \ar[u]^= &\ar[l]\dvr\ar[u] &\ar[l]_{\unif}
            \dvr\ar[u]_{\text{onto}}&\ar[l]0}
  \]
  The diagram implies that as an \dvr-module, $M$ is $\dvr/(\unif^2)$.
  In particular the extension on the top line is nontrivial over \dvr,
  which, by backing up the exact triangle as in
  \ref{MappedCofibrations}, implies that $a^*\colon \Sigma^{-1}X\to X$
  represents a nonzero element of $H_{-1}\End_{\EE}(X)$. But, by
  construction, $a^*$ is given by left multiplication with a generator
  of $H_{-1}\dga$.  \qed
\end{titled}

\begin{numbered}{Eilenberg-Moore spectral sequence}\label{ExplainEM}
If \dga/ is a \DGA/ and $X$, $Y$ are modules over \dga, there is an
Eilenberg-Moore spectral sequence
\[
     E^2_{-i,j}=\Ext^i_{H_*\dga}(\Sigma^jH_*X, H_*Y)\Rightarrow
     H_{j-i}\Hom_{\dga}(X,Y)\,.
\]
In a homotopy context this can be constructed in precisely the same
way as an Adams spectral sequence. In an algebraic context it is
constructed by inductively building an exact sequence of \dga-modules
\begin{equation}\label{EMResolution}
\xymatrix@1{0&\ar[l]X& \ar[l] F(0)& \ar[l] F(1) &\ar[l]\cdots &\ar[l]
  F(i)&\ar[l]\cdots}
\end{equation}
such that 
\begin{itemize}
\item each $F(i)$ is a a sum of shifts of
  copies of \dga, and
\item applying $H_*$ to \ref{EMResolution} produces a free resolution
  of $H_*X$ over $H_*\dga$.
\end{itemize}
See for 
instance \cite[9.11]{rLooking}. 
The totalization $tF$ of the double complex $F$ is then a project (or cofibrant)
model for $X$; filtering $tF$ by $\{tF({\le} n)\}_n$ and applying
$\Hom_{\dga}(\whatever, Y)$ gives  a filtration of
$\Hom_{\dga}(X,Y)$ which yields the spectral sequence.
  \end{numbered}

\section{Exterior algebras over $\Z$}\label{CUniqueOverZ}
In this section we prove the first claim of Theorem \ref{ModZcase}. To
be specific, we show that any \DGA/ \dga/ of type \Br\Z/  is equivalent to the formal \DGA/ \FF/ given by the following chain
complex concentrated in degrees $0,\,-1$.
\[
 \xymatrix@C=6pt{{\phantom{-}0} & \Z \ar[d]^{\partial=0}\\
                 -1 & \Z}
\]
The multiplication on \FF/ is the only
possible one consistent with the requirement that $1\in\Z=\FF_0$ act
as a unit.

We will only sketch the line of reasoning, since it is similar to that
in sections \S\ref{CSemiCanon}-\ref{CExtOverFp}, although the
conclusion is very different. Along the lines of
\S\ref{CSemiCanon} there exists a module $X$ of type $(\Z,0)$ over
\dga, and this is  unique up to noncanonical \equivalence. Let
$\EE=\End_{\dga}(X)$. The argument in the proof of \ref{DGAtoRing}(1)
shows that $H_i\EE$ vanishes for $i\ne0$, and that $H_0\EE$ is a ring
\ring/ with a decreasing sequence of ideals
\[ 
    \cdots\subset
    \ideal_k\subset\cdots\subset\ideal_2\subset\ideal_1\subset\ring
  \]
such that $\ideal_k\ideal_\ell\subset\ideal_{k+\ell}$,
  $\Gr(\ring)\cong\Z[t]$, and $\ring\cong\lim_k\ring/\ideal_k$. Let
  $\sigma\in \ideal_1$ be an element which projects to a generator of
  $\ideal_1/\ideal_2\cong \Z$, and map $\Z[s]$ to \ring/ by sending
  $s$ to $\sigma$. It is easy to show by induction on $k$ that this
  map induces isomorphisms $\Z\poly s/(s^{k+1})\to R/\ideal_k$, and so
  induces an isomorphism $\Z\pwr s\to R$. Using the free resolution
  \begin{equation}\label{SmallResOverZ}
    \xymatrix@1{0\ar[r] & {\Z\pwr s}\ar[r]^s &{\Z\pwr s}\ar[r]&\Z\ar[r]&0}
  \end{equation}
  of $\Z$ over $\Z\pwr s$ ($s$ acting by zero on \Z) one argues as in the proof of
  \ref{RingtoDGA}(1) that $\End_{\Z\pwr s}(\Z)$ is a
  \DGA/ of type $\Br \Z$.  As in 
 the proof of \ref{DGAtoRing}(2), the natural map 
  $A\to\End_{\EE}(\dgm)\sim\End_{\Z\pwr s}(\Z)$ is
  \anequivalence. This reasoning applies in particular to \FF/,
   giving $\FF\sim\End_{\Z\pwr s}(\Z)$. Hence \dga/ is
  equivalent to \FF.
\section{Non-canonical modules of type \pzero/}\label{CManyModules}

In this section we complete the proof of \ref{ModZcase}. The
uniqueness statement for \dga/ was proven in \S\ref{CUniqueOverZ}, so
we only have to handle the classification of modules of type~\pzero. 

Let
$\ring=\Z\poly s$ and let $\ring$ act on $\Z$ via $\Z\cong\ring/(s)$. Clearly $\Ext^i_{\ring}(\Z,\Z)$ is $\Z$ for $i=0,1$ and zero
otherwise, so that $\EE=\End_{\ring}(\Z)$ is a \DGA/ of type
$\Br\Z$. By uniqueness, we may as well take $\dga=\EE$; the following
proposition explains why this is useful.

Say that a chain complex \NN/ of \ring-modules is \emph{$s$-torsion} if
for each $i\in\Z$ and each $x\in H_i\NN$ there is an integer $k(x)>0$
such that $s^{k(x)}x=0$. 

\begin{prop}\cite[2.1]{CompleteTorsion}\label{ConvertToR}
  The assignment $\NN\mapsto\Hom_R(\Z,N)$ restricts to a bijection
  between equivalence classes of $s$-torsion chain complexes over $R$
  and equivalence classes of right \EE-modules.
\end{prop}

\begin{rem} Here \EE/ acts on $\Hom_R(\Z,N)$ through its action
  on \Z.  The inverse to this bijection assigns to a right \EE-module
  \dgm/ the derived tensor product $\dgm\Tensor_{\EE}\Z$.
\end{rem}

It is clear from \S\ref{CUniqueOverZ} that  $\EE\sim\FF$ is equivalent as a \DGA/ to
  its opposite algebra, so we can pass over the distinction
  between right and left \EE-modules. The question of studying
  \EE-modules of type \pzero/ thus becomes one of
  classifying $s$-torsion chain complexes \NN/ over \ring/ such that
  \begin{equation}\label{MagicCondition}
             H_i\Hom_{\ring}(\Z,\NN)\cong\begin{cases}
                                       \Zmodp & i=0\\
                                         0    & i\ne 0
                                       \end{cases}
\end{equation}
Suppose that \NN/ is such a chain complex.
Applying $\Hom_R(\whatever,\NN)$ to the exact sequence
\[
 \xymatrix@1{0\ar[r] & R\ar[r]^s& R\ar[r]&\Z\ar[r]&0}
\]
and taking homology gives a long exact sequence
\[
\xymatrix@1{\cdots\ar[r]&H_i\Hom_{\ring}(\Z,\NN)\ar[r]
  &H_i\NN\ar[r]^s&H_i\NN\ar[r]&\cdots}\,.
\]
in which the arrow labeled ``$s$'' cannot be injective unless
$H_i\NN=0$. 
It follows that $H_i\NN$ vanishes unless $i=0$, and that $H_0\NN$ is
an abelian group \DD/ on which the operator $s$ acts with the
following properties:
\begin{enumerate}
\item $\ker(s)$ is isomorphic to $\Z/p$,
\item $s$ is surjective, and
\item $\cup_{k\ge0}\ker (s^k)= \DD$.
\end{enumerate}
Call such a pair $(\DD, s)$ \emph{\admissible/}. In fact, any
\admissible/ pair
$(\DD,s)$ gives a chain complex \NN/ over \ring/ (\DD/ itself
concentrated in degree~0) which has property \ref{MagicCondition}.
Combining this observation with \ref{SometimesUnique} and
\ref{ConvertToR} thus provides a bijection between
\begin{itemize}
\item equivalence classes of modules over \EE/ of type \pzero, and
\item isomorphism classes of \admissible/ pairs $(\DD,s)$.
\end{itemize}

The proof of \ref{ModZcase} is completed by the following two routine
lemmas. Recall that $\ring=\Z\poly s$. If \dvr/ is a discrete valuation
ring with uniformizer \unif, write $\dvr/\unif^\infty$ for the
quotient $\dvr[1/\unif]/\dvr$. (This quotient is the injective hull of
$\dvr/\unif$ as an ordinary module over \dvr.)

\begin{lem}\label{StartWithPair}
  Suppose that $(\DD,s)$ is \adistinguished/ pair. Then
  $\dvr=\Ext^0_{\ring}(\DD,\DD)$ is a complete discrete valuation ring
  with residue field \Fp/ and uniformizer $\unif=s$. The pair $(D,s)$
  is naturally isomorphic to $(\dvr/\unif^\infty,\pi)$.
\end{lem}

\begin{lem}\label{StartWithDVR}
  Suppose that \dvr/ is a complete discrete valuation ring with
  residue field \Fp/ and uniformizer~\unif. Then
  $(\DD,s)=(\dvr/\pi^\infty,\pi)$ is \adistinguished/ pair, and the
  natural map $\dvr\to\Ext^0_R(\DD,\DD)$ is an isomorphism.
\end{lem}

\begin{rem}\label{ChainComplexFromModule}
  According to \ref{ModZcase}, any complete discrete valuation ring
  \dvr/ with residue field \Fp/ and uniformizer \unif/ should give
  rise to a module \dgm/ of type \pzero/ over the formal \DGA/ \FF/ of
  \S\ref{CUniqueOverZ}. Observe  that a module over \FF/ is just a chain complex with a
  self-map $f$ of degree~$-1$ and square~$0$. Let $(D,s)=(\dvr/\unif^\infty,\unif)$ be the
  \distinguished/ pair associated to $\dvr$ and~\unif. Tracing through
  the above arguments shows that $X$ can be taken to be the following
  object
  \[
   \xymatrix@C=40pt{{\phantom{\hbox{$-$}}}0 & D \ar[d]^{\partial=s}
     \ar@/_1.5pc/[d]_{f=\operatorname{id}}\\
             -1 & D}
  \]
  concentrated in degrees~$0$ and~$-1$.
\end{rem}

\section{Other exterior algebras}\label{COther}
Suppose that \ring/ is a commutative ring, and say that a \DGA/ \dga/ is
\emph{of type \BrOf \ring n/} if $H_*\dga$ is an exterior algebra over \ring/ on a
class of degree~$n$. In this section we briefly consider the problem
of classifying such \DGAs/  if $n\ne -1$. If $n=0$,
\dga/ is formal (\ref{Formality}) and hence determined up to equivalence by
the ring $H_0\dga$, so we may as well also  assume~$n\ne 0$.

Say that a \DGA/ is of type \PrOf \ring n/ if $H_*\dga$ is isomorphic
to a polynomial algebra over \ring/ on a class of degree~$n$.

\begin{prop} \label{SimpleDuality} If $n\notin\{0,-1\}$, there is a natural bijection between
  equivalence classes of \DGAs/ of type \BrOf \ring n/ and equivalence
  classes of \DGAs/ of type \PrOf\ring{-n-1}/.
\end{prop}

\begin{proof}(Sketch)  If \dga/ is of type \BrOf \ring n/, the inductive
  technique of \S\ref{CSemiCanon} produces a left module $\dgm_{\dga}$ over
  \dga/ of type $(R,0)$, i.e., a module $\dgm_{\dga}$ such that
  $H_0(\dgm_{\dga})$ is a free (ordinary) module of rank~1 over
  $R=H_0\dga$ and $H_i(\dgm_{\dga})=0$ for $i\ne0$.  If \dgab/ is of type
  \PrOf\ring n/, it is even easier to produce a right module $\dgm_{\dgab}$
  over $\dgab$ of type $(R,0)$: just take the mapping cone of
  $f\colon\Sigma^n\dgab\to\dgab$, where $f$ represents
  left multiplication by a generator of $H_n\dgab$. In both cases the
  modules are unique up to possibly non-canonical
  equivalence. Along the lines of \S\ref{CExtOverFp}
  (cf. \ref{MooreKoszul}), calculating with appropriate collapsing Eilenberg-Moore
  spectral sequences now gives the desired bijection.
\[
\xymatrix@1@C=4.5pc{\txt<13pc>{\{\DGAs/ \dga/ of type \BrOf\ring n/} \ar@<0.5ex>[r]^-{\dga\mapsto\End_{\dga}(\dgm_{\dga})} &\quad\txt<13pc>{\{\DGAs/ \dgab/ of type \PrOf\ring {-n-1}/\}}\ar@<1.5ex>[l]^-{\dgab\mapsto\End_{\dgab}(\dgm_{\dgab})}}
\]
\end{proof}

\begin{rem} It is clear from \ref{Modpcase} that Proposition \ref{SimpleDuality} fails for $n=-1$ and
  $\ring=\Fp$, essentially because if $\dga$ is of type $\BrOf\Fp{-1}$
  the nonvanishing entries of the Eilenberg-Moore spectral sequence
  for $H_*\End_{\dga}(\dgm_{\dga})$ accumulate in degree~0.  This
  accumulation creates extension possiblities which allow for a
  profusion of complete
  discrete
  valuation rings in the abutment. Similarly, \ref{SimpleDuality} fails for $n=0$ and
  $\ring=\Fp$, because if \dgab/ is of type $\PrOf\Fp{-1}$ the
  Eilenberg-Moore spectral sequence for
  $H_*\End_{\dgab}(\dgm_{\dgab})$ accumulates in degree~0. In this
  case, though, the accumulation has consequences which are less
  drastic, because up to isomorphism there are only two possibilities
  for a ring $R$ with an ideal $I$ such that $I^2=0$ and
  such that the associated graded ring $\{R/I, I\}$ is an exterior
  algebra on one generator over~\Fp.  The conclusion is that up to equivalence there are only two \DGAs/ of type
  $\PrOf\Fp{-1}$, one corresponding to the true exterior algebra
  $\Fp[t]/t^2$, and the other to the fake exterior algebra $\Z/p^2$.
\end{rem}

\begin{numbered}{\DGAs/ of type \BrOf\Z n/, all~$n$} Up to equivalence, there
  is only one of these for each $n$. The case $n=0$ is trivial
  (\ref{Formality}), while $n=-1$ is
  \ref{ModZcase}. By \ref{SimpleDuality}, for other $n$ these
  correspond to \DGAs/ of type $\PrOf\Z n$, but there is only
  one of these for each $n$, because they are all formal~(\ref{Formality}). 
\end{numbered}

\begin{numbered}{\DGAs/ of type $\BrOf\Fp n$, $n\ge0$} As usual, the
  case $n=0$ is trivial: up to equivalence there is only one
  example. We sketch an argument that if $n>0$ is odd there is only
  one example, while if $n>0$ is even, there are two. In
  \cite{Postnikov}, Dugger and Shipley describe a Postnikov approach
  to constructing a connective (\ref{SometimesUnique}) \DGA/ \dga; the
  technique involves starting with the ring $H_0\dga$ (considered as a
  \DGA/ with trivial higher homology) and attaching one homology group
  at a time, working from low dimensions to high. If \dga/ is of type
  $\BrOf\Fp n$ there is only a single homology group to deal with,
  namely $\Z/p$ in degree~$n$. By \cite[Thm.~8]{Postnikov} and a
  theorem of Mandell \cite[Rem.~8.7]{Postnikov} the choices involved
  in attaching $H_n\dga$ can be identified with the group
  $HH^{n+2}_{\Z}(\Fp,\Z/p)$; this is Shukla cohomology of \Fp/ with
  coefficients in the $\Fp$-bimodule $\Z/p$. In our notation this
  group might be written
 \begin{equation}\label{Gluing}
              H_{-n-2}\Hom_{\Fp\Tensor_{\Z}\Fp\op}(\Z/p,\Z/p)
 \end{equation}
 where the indicated tensor product over \Z/ is derived. (This is the
 appropriate variant of Hochschild cohomology when the ring involved,
 here~\Fp, is not flat over the ground ring, here~\Z.) The ring
 $H_*(\Fp\Tensor_{\Z}\Fp\op)$ is an exterior algebra over $\Fp$ on a class
 of degree~$1$, so the Eilenberg-Moore spectral sequence computes
  \[
    H_*\Hom_{\Fp\Tensor_{\Z}\Fp\op}(\Z/p,\Z/p)\cong \Fp[u]
\]
 where the degree of~$u$ is $-2$. This immediately shows that the
 group \ref{Gluing} of gluing choices is trivial if $n$ is odd. If $n$
 is even, there are $p$ gluing choices, but $p-1$ of them are
 identified by the automorphisms of $\Z/p$ as an $\Fp$-bimodule. 
The conclusion is that
 if $n$ is even there are $p$ gluing choices, but that up to
 equivalence only two \DGAs/ emerge.

 By \ref{SimpleDuality} this also gives a classification of \DGAs/ of
 type $\PrOf\Fp n$ for $n\le-2$. 
\end{numbered}

We do not know how to classify \DGAs/ of type $\BrOf\Fp n$ for
$n\le-2$. 

\providecommand{\bysame}{\leavevmode\hbox to3em{\hrulefill}\thinspace}


\begin{thebibliography}{10}

\bibitem{rLooking}
L.~Avramov and S.~Halperin, \emph{Through the looking glass: a dictionary
  between rational homotopy theory and local algebra}, Algebra, algebraic
  topology and their interactions ({S}tockholm, 1983), Lecture Notes in Math.,
  vol. 1183, Springer, Berlin, 1986, pp.~1--27.

\bibitem{Bourbaki}
N.~Bourbaki, \emph{Commutative algebra. {C}hapters 1--7}, Elements of
  Mathematics (Berlin), Springer-Verlag, Berlin, 1998, Translated from the
  French, Reprint of the 1989 English translation.

\bibitem{Postnikov}
D.~Dugger and B.~Shipley, \emph{Postnikov extensions of ring spectra}, Algebr.
  Geom. Topol. \textbf{6} (2006), 1785--1829 (electronic).

\bibitem{CompleteTorsion}
W.~G. Dwyer and J.~P.~C. Greenlees, \emph{Complete modules and torsion
  modules}, Amer. J. Math. \textbf{124} (2002), no.~1, 199--220.

\bibitem{Duality}
W.~G. Dwyer, J.~P.~C. Greenlees, and S.~Iyengar, \emph{Duality in algebra and
  topology}, Adv. Math. \textbf{200} (2006), no.~2, 357--402.

\bibitem{rMoore}
J.~C. Moore, \emph{Differential homological algebra}, Actes du {C}ongr\`es
  {I}nternational des {M}ath\'ematiciens ({N}ice, 1970), {T}ome 1,
  Gauthier-Villars, Paris, 1971, pp.~335--339.

\bibitem{rSchwedeShipleyModules}
S.~Schwede and B.~Shipley, \emph{Stable model categories are categories of
  modules}, Topology \textbf{42} (2003), no.~1, 103--153.

\bibitem{Serre}
J.-P. Serre, \emph{Local fields}, Graduate Texts in Mathematics, vol.~67,
  Springer-Verlag, New York, 1979, Translated from the French by Marvin Jay
  Greenberg.

\bibitem{ShipleyHZ}
B.~Shipley, \emph{{$H\Bbb Z$}-algebra spectra are differential graded
  algebras}, Amer. J. Math. \textbf{129} (2007), no.~2, 351--379.

\bibitem{rWeibelBook}
C.~A. Weibel, \emph{An introduction to homological algebra}, Cambridge
  University Press, Cambridge, 1994.

\end{thebibliography}
\end{document}